\documentclass{birkjour}
 \newtheorem{thm}{Theorem}[section]
 
 \newtheorem{lem}[thm]{Lemma}
 
 \theoremstyle{definition}
 
 \theoremstyle{remark}

 \numberwithin{equation}{section}

\newcommand{\field}[1]{\mathbb{#1}}
\newcommand{\MM}{\field{M}}
\newcommand{\SL}{\mathrm{SL}}
\newcommand{\pSL}{\mathrm{(P)SL}}
\newcommand{\Z}{\mathrm{Z}}

\begin{document}
\title[Special rank one groups]
{Special rank one groups are perfect}

\author[Anja Steinbach]{Anja Steinbach}

\address{%
Justus-Liebig-Universit\"at Gie{\ss}en\\
Mathematisches Institut\\
Arndtstra{\ss}e 2\\
D 35392 Gie{\ss}en\\
Germany}

\email{Anja.Steinbach@math.uni-giessen.de}



\date{October 7, 2011}

\begin{abstract}
We prove that special (abstract) rank one groups with arbitrary unipotent subgroups of size at least 4 are perfect.
\end{abstract}

%%% ----------------------------------------------------------------------
\maketitle
%%% ----------------------------------------------------------------------
%\tableofcontents
\section{Introduction}

J. Tits \cite{Tits-Durham} defined Moufang sets in order to axiomatize the linear algebraic groups of relative rank one.
A closely related concept of so-called (abstract) rank one groups has been introduced by F. G. Timmesfeld \cite{TT-Buch}.

Here a group $X$ is an (abstract) rank one group with unipotent subgroups $A$ and $B$, 
if $X = \langle A,B\rangle$ with $A$ and $B$ different subgroups of $X$, and (writing $A^b = b^{-1}Ab$)
\begin{align*}
&\text{for each $1 \neq a \in A$, there is an element $1 \neq b \in B$ such that $A^b = B^a$,} \\
&\text{and vice versa.}
\end{align*}
We emphasize that in contrast to Timmesfeld's definition \cite[p. 1]{TT-Buch} we do not assume that $A$ and $B$ are nilpotent. In an (abstract) rank one group $X$ with unipotent subgroups $A$ and $B$,
the element $1 \neq b \in B$ with $A^b = B^a$ is uniquely determined for each $1 \neq a \in A$ (as $A \neq B$) and denoted by $b(a)$. Similary we define $a(b)$.

We say $X$ is special, if $b(a^{-1}) = b(a)^{-1}$ for all $1 \neq a  \in A$. This is equivalent with Timmesfeld's original definition, see Timmesfeld \cite[I (2.2), p. 17]{TT-Buch}.
In this note we prove 
\begin{thm}\label{Main-Result}
Any special (abstract) rank one group with arbitrary unipotent subgroups of size at least $4$ is perfect.
%Let $X$ be a special (abstract) rank one group with uni\-potent subgroups $A$ and $B$ 
%and set $H = N_X(A) \cap N_X(B)$. Then $A = [A,H]$ provided that $|A| \geq 4$.
\end{thm}

By \cite[I (1.10), p. 13]{TT-Buch} this yields that a special (abstract) rank one group with abelian unipotent subgroups is either quasi-simple or isomorphic to $\SL_2(2)$ or $\pSL_2(3)$, as was conjectured by Timmesfeld \cite[Remark, p.~26]{TT-Buch}. 

We remark that $X / \Z(X)$ is the little projective group of a Moufang set and that is the point of view 
of T. De Medts, Y. Segev and K. Tent. In \cite[Theorem 1.12]{DM-S-T} they prove that the little projective group $G$ of a special Moufang set $\MM(U, \tau)$ with $|U| \geq 4$ satisfies $U_\infty = [U_\infty, G_{0,\infty}]$. From this they deduce Theorem \ref{Main-Result} above.

%, who proved Timmesfeld's conjecture first. They deduced it from suitable identities in special Moufang sets $\MM(U, \tau)$, see \cite[Theorem 1.12]{DM-S-T}, without the assumption that $U$ is abelian.

%In Timmesfeld's (quasi) simplicity criterion for groups generated by abstract root subgroups \cite[II (2.14)]{TT-Buch}
%only (abstract) rank one groups with abelian unipotent subgroups are involved. My aim was to simplify his criterion via a short, elementary and self-contained proof of Theorem \ref{Main-Result}. 

The proof of Theorem \ref{Main-Result} given below is short, elementary and self-contained. It does not need a case differentiation whether $A$ is an elementary abelian 2-group or not. We show that $a_1(A \cap X') = a_2(A \cap X')$ 
for all $1 \neq a_1, a_2 \in A$ with $a_1a_2 \neq 1$. (Here $X' = [X,X]$ is the commutator subgroup of $X$.)

\section{The proof of Theorem \ref{Main-Result}}
Let $X$ be a special (abstract) rank one group with unipotent subgroups $A$ and $B$.
For each $1 \neq a \in A$, we set $n(a) := ab(a)^{-1}a$. 
Then $B^{n(a)} = A$. As $b(a)^{-1} = b(a^{-1})$, also $A^{n(a)} = B$. Thus $n(a)n(a') \in H := N_X(A) \cap N_X(B)$, for
all $1 \neq a, a' \in A$.
For $1 \neq a \in A$, we have
\begin{equation}\label{force-h}
B^{a^{-1}n(a)} = B^a
\end{equation}

\begin{lem} \label{useful-equation}
We have $B^{a_1a_2n(a_2^{-1})a_2} = 
%B^{a_2a_1n(a_1^{-1})a_1}$
A^{b(a_1)b(a_2)}$, for all $1 \neq a_1, a_2 \in A$.
\end{lem}

\begin{proof}
By the definition of $n(a_2^{-1})$ we have $B^{a_1a_2n(a_2^{-1})a_2} = B^{a_1b(a_2^{-1})^{-1}}$.
As $X$ is special, the left hand side of the claim equals $A^{b(a_1)b(a_2)}$.
%Similarly, the right hand side is $A^{b(a_2)b(a_1)}$. 
%As $B$ is abelian, the claim follows.
\end{proof}

\begin{lem} \label{equation-in-A}
We have $a_1 \in a_2(A \cap X')$, for all $1 \neq a_1, a_2 \in A$ with $a_1a_2 \neq 1$.
\end{lem}

\begin{proof}
Let $1 \neq a_1, a_2 \in A$ with $a_1a_2 \neq 1$. 
We set $h := n(a_1a_2)n(a_2^{-1}) \in H$.
By Lemma \ref{useful-equation} and \eqref{force-h} we have $R := A^{b(a_1)b(a_2)} = B^{a_2^{-1} a_1^{-1} h a_2}$.
As $a_2^{-1} a_1^{-1} h a_2 = h a_1^{-1}[a_1^{-1},a_2][a_2^{-1}a_1^{-1}a_2,h][h,a_2]$ and $B^h = B$, 
we obtain that $R = B^{a_1^{-1}a_0}$ with $a _0 \in A \cap X'$.

Note that $b(a_1)b(a_2) = b(a_2)b(a_3)$, where $1 \neq a_3 = a(b_3)\in A$ with 
$1 \neq b_3 = b(a_2)^{-1}b(a_1)b(a_2) \in B$.
Necessarily $a_2a_3 \neq 1$. Otherwise Lemma \ref{useful-equation} implies that
$R = A^{b(a_1)b(a_2)} = A^{b(a_2)b(a_3)} = B^{n(a_3^{-1})a_3} = A$, a contradiction as $R = B^a$ with $a \in A$.
As above we have $R = A^{b(a_2)b(a_3)} = B^{a_2^{-1}a_4}$ with $a_4 \in A \cap X'$.
As $N_A(B) = 1$, we obtain $a_1^{-1}(A \cap X') = a_2^{-1}(A \cap X')$. Thus the claim holds.
\end{proof}

When $A \cap X' = 1$, then Lemma \ref{equation-in-A} implies that 
$A \subseteq \{1, a, a^{-1}\}$, where $1 \neq a \in A$; i.e., $|A| \leq 3$.
Thus for $|A| \geq 4$, we may choose $1 \neq a \in A \cap X'$. By Lemma \ref{equation-in-A}, we obtain
$A \subseteq a(A \cap X') \cup \{1, a^{-1}\} \subseteq A \cap X'$, as desired.

% ------------------------------------------------------------------------

\begin{thebibliography}{1}
\bibitem{DM-S-T} T. De Medts, Y. Segev and K. Tent, \textit{Special Moufang sets, their root groups and their $\mu$-maps.} Proc. Lond. Math. Soc. (3) {\bf 96} (2008), 767-–791. 
\bibitem{TT-Buch} F. G. Timmesfeld, \textit{Abstract root subgroups and simple groups of Lie type.} 
Monographs in Mathematics, {\bf 95}. Birkh\"auser, 2001.
\bibitem{Tits-Durham} J. Tits, \textit{Twin buildings and groups of Kac-Moody type.} Groups, combinatorics \& geometry (Durham, 1990), 249-–286, London Math. Soc. Lecture Note Ser., {\bf 165}, Cambridge Univ. Press, Cambridge, 1992. 
\end{thebibliography}
\end{document}